\documentclass[a4paper, 11 pt, twoside]{amsart}
\usepackage{srollens-en}[2011/03/03]
\usepackage[left = 3.5cm, right = 3.5cm, headsep = 6mm,
footskip = 10mm, top = 35mm, bottom = 35mm, footnotesep=5mm, headheight =
2cm]{geometry}

\usepackage{enumitem}
\newlist{prooflist}{description}{1}
\setlist[prooflist]{font=\normalfont \itshape, labelindent = \parindent, leftmargin = 0pt}

\usepackage{soul,xcolor}

\setstcolor{blue}

\usepackage{comment}

\makeatletter
\g@addto@macro\@floatboxreset\centering
\makeatother

\usepackage[unicode,bookmarks, pdftex]{hyperref}
\hypersetup{colorlinks=true,citecolor=NavyBlue,linkcolor=NavyBlue,urlcolor=Orange, pdfpagemode=UseNone, breaklinks=true}

\renewcommand{\lg}{\mathfrak{g}}

\newcommand{\lh}{\mathfrak{h}}

\newcommand{\nulleins}[1]{{#1}^{0,1}}
\newcommand{\einsnull}[1]{{#1}^{1,0}}

\newcommand{\liesl}{\mathfrak{sl}}

\DeclareMathOperator{\ad}{\mathsf{ad}}

\title{Almost abelian complex nilmanifolds}

\author[A. Andrada]{Adrián Andrada}
\author[R. M. Arroyo]{Romina M. Arroyo}
\author[M. L. Barberis]{María L. Barberis }
\address{Adrián Andrada, Romina M. Arroyo and María L. Barberis  \newline
 FAMAF, Universidad Nacional de C\'ordoba and CIEM, CONICET
 \newline Av. Medina Allende s/n, Ciudad Universitaria \\ 
 X5000HUA C\'ordoba\\ Argentina}

\email{adrian.andrada@unc.edu.ar}
\email{romina.arroyo@unc.edu.ar}
\email{mlbarberis@unc.edu.ar}

\author[S. Rollenske]{S\"onke Rollenske}
\author[K. Wehler]{Konstantin Wehler}
\address{S\"onke Rollenske, Konstantin Wehler
\newline
FB 12/Mathematik und Informatik\\
Philipps-Universit\"at Marburg\newline
Hans-Meerwein-Str. 6\\
35032 Marburg\\Germany}

\email{rollenske@mathematik.uni-marburg.de}
\email{konstantin.wehler@uni-marburg.de}

\begin{document}
\begin{abstract}
We show that a complex structure on a nilpotent almost abelian real Lie algebra is unique if it exists. As a consequence, we get full control over the cohomology and deformations of almost abelian complex nilmanifolds.
\end{abstract}

\maketitle
\setcounter{tocdepth}{1}
\tableofcontents

\section{Introduction}
Nilmanifolds have been established as one of the most reliable classes of manifolds if one is looking for an accessible example of an exotic geometric structure. But since nilpotent Lie algebras are classified only up to dimension $7$, and even the theory of $2$-step nilpotent Lie algebras is very rich,  it is much rarer to have general results.

In this work, we are interested in what is usually called a complex nilmanifold, that is, a compact quotient of a real nilpotent Lie group $G$ by a lattice $\Gamma$,  equipped with a complex structure induced by a  left-invariant complex structure on $G$. The latter is uniquely determined by an almost complex structure $J\colon \lg \to \lg$ on the Lie algebra $\lg$ of $G$ satisfying an integrability condition. 
For such a manifold one hopes to prove that, as much as possible, its geometry and its topological and holomorphic invariants are determined by $(\lg, J)$ and are independent of the lattice; compare \cite{rollenske11}.
This has been successful for particular kinds of complex structures \cite{Sakane, CFP06, rtx20} and in some cases with very small commutator \cite{rollenske09}. This includes all cases in  dimension six, where a rough classification of possible $(\lg, J)$ was provided by Salamon \cite{salamon01}, which was later refined in \cite{ceballos2016invariant}. Another important feature of complex nilmanifolds, which is independent of the lattice, is that they have trivial canonical bundle \cite{BDV}. 

The aim of this article is to provide a nearly complete understanding of complex nilmanifolds whose underlying Lie algebra is almost abelian, that is, it admits an abelian ideal of codimension one. We will call such manifolds almost abelian complex nilmanifolds. Note that neither the nilpotency index nor the dimension of the commutator is bounded in this class. More precisely, if the nilpotent almost abelian Lie algebra has dimension $2n$ then the step of nilpotency is at most $n$ and the dimension of the commutator ideal is at most $2n-2$.

Building on the work of \cite{ABDGH24}, we show that if a complex structure exists on a nilpotent almost abelian Lie algebra it is unique up to Lie algebra isomorphism (Proposition \ref{prop: classification}). Analysing this particular complex structure and using the standard tools of the trade, we take full geometric and cohomological control:
\begin{thm}\label{thm: stbs}
  Let $X = (\Gamma\backslash G, J)$ be an almost abelian complex nilmanifold with associated Lie algebra $\lg$. Then the following hold:
  \begin{enumerate}
      \item The Lie algebra $\lg$ admits a stable torus bundle series and thus $X$ has the structure of an iterated principal holomorphic torus bundle as in \eqref{eq: tower nilmanifold} below.
   \item The Dolbeault cohomology of $X$ is computed by left-invariant forms, that is, the natural inclusion 
   \[ H^{p,q}(\lg, J) \into H^{p,q}_{\delbar}(X)\]
   is an isomorphism.
   \item The Frölicher spectral sequence of $X$ degenerates at the first page.
  \end{enumerate}
\end{thm}
Although the proof of the last item is slightly more involved, as a byproduct, we obtain an explicit procedure to compute all Betti and Hodge numbers of such a manifold, as explained in Section \ref{sect: computation}. 

With the result above as a starting point, it is natural to look at families, respectively, the space of all (left-invariant) complex structures on the nilmanifold. Again, the class of almost abelian complex nilmanifolds appears to be made to fit the assumptions of the already existent results, giving us the strongest possible conclusion.
\begin{thm}\label{thm: deform}
Let $X$ be an almost abelian complex nilmanifold. Then the following hold:
\begin{enumerate}
     \item Every small deformation of $X$ is a complex nilmanifold and deformations of $X$ are unobstructed, that is, the Kuranishi space is smooth. Moreover, the Kuranishi family is universal.
   \item Every deformation in the large of $X$ is a complex nilmanifold, that is, if $\kx \to B$ is a smooth family over a connected base and one fibre is isomorphic to $X$, then every fibre is isomorphic to a complex nilmanifold (with the same underlying real nilmanifold).
\end{enumerate}
\end{thm}
To our knowledge, this is the first class of complex nilmanifolds with arbitrarily large nilpotency index, where one has this kind of control over the cohomology and deformations. It would be interesting to investigate whether the strong complex geometric properties are reflected in the existence of special Hermitian metrics.

Some results on this topic are already available in the literature. For instance, it was shown in \cite{AL} that the only nilpotent almost abelian Lie algebras admitting an SKT metric are Lie algebras isomorphic to the direct sum of the three-dimensional Heisenberg algebra and an abelian Lie algebra. In fact, the existence of an SKT metric on a nilpotent Lie algebra implies that the Lie algebra has to be 2-step nilpotent (see \cite{AN}).

Other types of geometric structures have been considered on almost abelian nilmanifolds (and in the more general class of almost abelian solvmanifolds). There is a vast literature on this subject, see for instance \cite{AB1, AB2, AT, BFLT, BF, FP2, Fr, LW, LRV, Mo}.

\subsection*{Acknowledgements}
We are grateful to 	Pierre-Emmanuel Chaput for some help regarding Proposition \ref{prop: counting} and also to Leandro Cagliero and Jorge Lauret for useful comments. 
This work was partially supported by the DFG through grants RO 3734/4-1 and AG 92/5-1. AA, RMA and MLB have been partially supported by CONICET, FONCYT and SeCyT-UNC. RMA also would like to acknowledge support from the ICTP through the Associates Programme (2023-2028).

\section{Preliminaries and notation} 

In this section, we recall some basic notions and collect some results on complex nilmanifolds. In particular, we fix the notation used throughout the rest of the paper.

\subsection{Complex nilmanifolds}
Let $\lg$ be a real nilpotent Lie algebra and $G$ its associated simply connected nilpotent Lie group. By a theorem of Malcev \cite{malcev51} the Lie group $G$ admits a lattice $\Gamma\subset G$, that is, a discrete co-compact subgroup, if and only if there is a Lie algebra $\lg_\IQ$ defined over $\IQ$ such that $\lg_\IQ \otimes \IR = \lg$. In other words, this means that $\lg$ admits a basis with rational structure constants. In this case, the quotient $M= \Gamma \backslash G$ is a compact nilmanifold.

Since we will be interested in nilpotent almost abelian Lie algebras, it is worth noting that the existence of a codimension one abelian ideal in $\lg$ together with the theorem mentioned above guarantees the existence of a lattice in $G$: using the Jordan normal form we can find a basis of $\lg$ with structure constants equal to 0 or 1 (see Section \ref{subsect: almost abelian}).

Let $J$ be an almost complex structure on $\lg$, that is, an endomorphism $J\colon \lg \to \lg$ such that $J^2 = -\id_\lg$. We say that $J$ is integrable if its Nijenhuis tensor vanishes on $\lg$ and call the pair $(\lg, J)$ a Lie algebra with complex structure. 

By left-translations every complex structure $J$ on $\lg$ can be extended to a complex structure on the Lie group $G$. Since this complex structure is left-invariant it descends to the quotient $M$ and $X= (M,J)$ is a compact complex nilmanifold.

\subsection{Lie algebra Dolbeault cohomology}\label{sect: Lie algebra cohomology}
In the following we briefly set up the necessary cohomology theory for Lie algebras with complex structure. For details we refer to \cite{rollenske09b}.

Recall that the cohomology $H^\bullet(\lg, V)$ of a Lie algebra $\lg$ with values in a $\lg$-module $V$ can be computed with the so-called Chevalley-Eilenberg complex \cite[Section~7]{Weibel}. Suppose $\lg$ admits a complex structure $J$. Then the complexified Lie algebra $\lg_\IC= \lg^{1,0} \oplus \lg^{0,1}$ decomposes into the $\pm i$-eigenspaces of the $\IC$-linear extension of $J$. Therefore, the exterior algebra of the dual space $\lg_\IC^*$ admits a decomposition
\[
\Wedge^k \lg_\IC^* = \bigoplus_{p+q=k} \Wedge^{p,q}\lg^*,
\]
where $\Wedge^{p,q}\lg^* = \Wedge^p\lg^{*1,0} \otimes \Wedge^q \lg^{*0,1}$. Since $J$ is integrable the differential $d$ in the Chevalley-Eilenberg complex $(\Wedge^\bullet \lg_\IC^*,d)$ of the trivial $\lg_\IC$-module $\IC$ decomposes into $d = \del + \delbar$. The adjoint action of $\lg$ induces a $\lg^{0,1}$-module structure on $\Wedge^p\lg^{*1,0}$ and the Dolbeault cohomology groups of $(\lg,J)$ can be defined as
\[
H^{p,q}(\lg,J) = H^q(\lg^{0,1}, \Wedge^p\lg^{*1,0}),
\]
where the Chevalley-Eilenberg complex computing $H^q(\lg^{0,1}, \Wedge^p\lg^{*1,0})$ is the complex $(\Wedge^{p,\bullet}\lg^*, \delbar)$.

Now consider  the geometric setting: if $X = (M,J)$ is a complex nilmanifold with Lie algebra $\lg$, then elements in $\Wedge^{p,q}\lg^*$ can be considered as left-invariant differential forms of type $(p,q)$ on $X$ and the differential $d= \del +\delbar$ on $\Wedge^{\bullet,\bullet}\lg^*$ coincides with the restriction of the usual differential on $X$ to left-invariant forms. 

It has been shown that the induced map
\[
H^{p,q}(\lg, J) \to H_{\delbar}^{p,q}(X)
\]
is always an inclusion \cite{con-fin01}. In many cases it is known to be an isomorphism; we say that Dolbeault cohomology can be computed by left-invariant forms. It is conjectured to always be an isomorphism (compare \cite{rollenske11} for an overview and \cite{frr19,rtx20} for the most recent progress).

\begin{rem}\label{rem: nomizu}
    Similarly, we have an inclusion $H^k(\lg,\IR)\hookrightarrow H_{dR}^k(M,\IR)$ of the Lie algebra cohomology of the trivial $\lg$-module $\IR$ into the de Rham cohomology of $M$. By a theorem of Nomizu \cite{nomizu54} this inclusion is always an isomorphism.
\end{rem}

\subsection{Fibrations and torus bundle series}
We recall in this section some observations about the geometry of complex nilmanifolds. 

Let $\lg$ be a real nilpotent Lie algebra. A rational structure for $\lg$ is a subalgebra $\lg_\IQ$ defined over $\IQ$ such that $\lg_\IQ \otimes \IR = \lg$. For instance, if $\Gamma \subset G$ is a lattice in the associated simply connected Lie group $G$, then the rational span $\langle \log \Gamma \rangle_\IQ \subset \lg$ is a rational structure for $\lg$ which determines the commensurability class of the lattice \cite[p.~204--205]{Cor-Green}. We say that a subalgebra $\lh\subset \lg$ is rational with respect to $\lg_\IQ$ if $\lh \cap \lg_\IQ$ is a rational structure for $\lh$.

There are two important filtrations of the Lie algebra $\lg$: the descending central series is defined by $\kc^{k+1}\lg = [\kc^k\lg, \lg]$ with $\kc^0\lg =\lg$ and the ascending central series is defined by $\kz^{k+1}\lg = \{ x \in \lg \,|\, [\lg,x] \subset \kz^k\lg\}$, where $\kz^0\lg =0$. In particular, the ideal $\kz^1\lg = \kz \lg$ is the centre of $\lg$. All of these subalgebras are rational with respect to every rational structure on $\lg$ \cite[p.~208]{Cor-Green}.

The following structure was first named in \cite{rollenske09}.
\begin{defin}\label{def: stbs} 
 Let $\lg$ be a real nilpotent Lie algebra with a fixed rational structure $\lg_\IQ$ and an integrable complex structure $J$. We call a filtration 
 \[0=\gothf^0 \subset \gothf^1 \subset \dots \subset \gothf^t=\lg\]
a  principal torus bundle series if for all $i=1, \dots, t$ the following conditions are satisfied:
\begin{enumerate}
 \item The subspaces $\gothf^i$ are rational.
 \item The subspaces $\gothf^i$ are $J$-invariant.
 \item We have $\gothf^{i+1}/\gothf^{i} \subset \kz \left(\lg/\gothf^i\right)$.
\end{enumerate}
The principal torus bundle series is called stable if the conditions are satisfied for every integrable complex structure $J$ and for every choice of rational structure on $\lg$. 
\end{defin}

The geometric significance of this notion is the following, extracted as a special case from \cite{rollenske09, rollenske11}.
\begin{prop}\label{prop: stbs}
Let $\lg$ be a real nilpotent Lie algebra admitting rational structures. Then $\lg$ admits a stable principal torus bundle series if and only if on the corresponding simply connected Lie group $G$ there is a tower of central extensions
\begin{equation}\label{eq: tower groups}
 \begin{tikzcd}
    G=G_0 \rar{\pi_0} & G_1 \rar{\pi _1} & {\cdots }\rar{\pi_{t-2}} & G_{t-1}
   \end{tikzcd} 
\end{equation}
such that for every choice of lattice $\Gamma\subset G$ and complex structure $J$ on $\lg$ the tower \eqref{eq: tower groups} induces a tower of holomorphic fibrations on $X = (\Gamma\backslash G, J)$, 
\begin{equation}\label{eq: tower nilmanifold}
 \begin{tikzcd}
    X= X_0 \rar{\pi_0} & X_1 \rar{\pi _1} & {\cdots} \rar{\pi_{t-2}} & X_{t-1}
   \end{tikzcd} 
\end{equation}
where each $\pi_i \colon X_i \to X_{i+1}$ is a principal holomorphic torus bundle over a complex nilmanifold of lower dimension. 
\end{prop}
This geometric description opens the door to all kinds of inductive arguments.

\begin{rem}
    Let $(\lg,J)$ be a nilpotent Lie algebra with complex structure. If $X=(M,J)$ is complex parallelisable, that is, $(\lg,J)$ is a complex Lie algebra, then the ascending and descending central series both provide a principal torus bundle series for $(\lg,J)$. If $J$ is abelian, that is, the subspace $\lg^{1,0}\subset \lg_\IC$ is abelian, then the ascending central series defines a principal torus bundle series for $(\lg,J)$.
    
    We will see in the next section that a (non-abelian) nilpotent almost abelian Lie algebra never admits a complex parallelisable structure and it admits an abelian complex structure only if it is isomorphic to the direct sum of the three-dimensional Heisenberg algebra $\lh_3$ and an abelian Lie algebra.
\end{rem}

\subsection{Remarks on the Jordan normal form}\label{sect: Jordan}
Let $V$ be a finite dimensional complex vector space of dimension $n$ and let $N$ be a nilpotent endomorphism of $V$.
 
 It is well known that there exists a unique partition  $n = \sum_i q_i \cdot i$ with $q_i \in \IN_0$  such that $N$ can be represented by a block diagonal, lower triangular matrix, where each block is a Jordan block of the form
  \begin{equation}\label{eq: J_k} \kj_k = \mat{  0 \\ 1 & \ddots&  \\  & \ddots & \ddots  \\0 &&1& 0 }\in \mathsf{M}_{k}(\IZ)
\end{equation}
 and $q_k$ gives the number of blocks of size $k$. 
 
Let us consider a slightly different perspective on this: by the Jacobson-Morozov theorem \cite[Theorem 10.3]{Knapp} we can extend $N$ to an $\liesl_2$-triple $H, N, M\in \End(V)$, making $V$ into an $\liesl_2(\IC)$-representation (see \cite[I.9]{Knapp} or \cite[Lecture 7]{FH}). Note that the choice of $M$ is not unique and everything that follows depends on this choice. 

 Now, a Jordan basis for $N$ can be found in the following way: decompose $V = \bigoplus_i V_i$ into irreducible $\liesl_2(\IC)$-representations. Then in $V_i$ choose a highest weight vector $v_i$, that is, a vector with $Mv_i =0$, and use  \[v_i,\; Nv_i, \dots,\; N^{\dim V_i -1} v_i\] 
 as part of the Jordan basis, giving the part of a basis corresponding to a Jordan block of size $\dim  V_i$.

 Having chosen the highest weight vectors $v_i$, we can write any vector $v\in V$ as
\[ v = \left(\sum_i \lambda_i v_i\right) + N w' = w + Nw'\]
with $Mw = 0 $.
If $w\neq0$, then  $w$ is the highest weight vector for an irreducible subrepresentation, that is, $w$ can be chosen as the starting vector for one Jordan block in a Jordan basis for $N$.

\section{Almost abelian complex nilmanifolds}\label{sect: almost abelian}
We now specialise to the class of almost abelian complex nilmanifolds.
\subsection{Almost abelian Lie algebras}\label{subsect: almost abelian}
A non-abelian Lie algebra $\lg$ is called \emph{almost abelian} if it has a codimension one abelian ideal $\gotha$. Similarly, we call a nilmanifold whose associated Lie algebra is almost abelian an \emph{almost abelian nilmanifold}. Suppose $\lg$ has dimension $2n+2$ and let $e_0\not\in \gotha$. Then the Lie algebra $\lg$ is isomorphic to the semi-direct product $\lg_A = \IR e_0 \ltimes_A \IR^{2n+1}$,
where $A$ is the matrix representing $\ad_{e_0}\colon \gotha \to \gotha$ with respect to some basis of $\gotha$. The Lie algebra $\lg_A$ is nilpotent if and only if the matrix $A$ is nilpotent and  two  almost abelian Lie algebras $\lg_A$ and $\lg_{A'}$ are isomorphic if and only if $A$ and $A'$ are conjugate, as explained in \cite{Fr}. Hence, by the existence of the Jordan normal form the isomorphism classes of nilpotent almost abelian Lie algebras of dimension $2n+2$ are in one-to-one correspondence with the partitions of $2n+1$. We call the Jordan partition of $A$ the Jordan partition of the almost abelian Lie algebra $\lg_A$. 

\begin{rem}\label{rem: fibre bundle}
    Let $\lg$ be a nilpotent almost abelian Lie algebra of dimension $2n+2$ which is not isomorphic to $\lh_3\oplus \IR^{2n-1}$. Then the centre $\kz\lg$ of $\lg$ has dimension at most $2n-1$ and it follows from the proof of \cite[Proposition~3.1]{hamrouni2008discrete} that the ideal $\gotha\subset \lg$ of codimension one is unique. Moreover, similar to \cite[Lemma~3.2]{hamrouni2008discrete} one verifies that $\gotha$ is also rational with respect to every rational structure on $\lg$. Therefore, every real nilmanifold with associated Lie algebra $\lg$ has the structure of a torus fibre bundle over the circle $S^1$. 
\end{rem}

\subsection{Nilpotent almost abelian Lie algebras with complex structure}
The question of existence of complex structures on nilpotent almost abelian Lie algebras was completely solved in \cite[Theorem~3.10]{ABDGH24}, which we rephrase slightly.
\begin{thm}[\cite{ABDGH24}]\label{thm: ABDGH}
 Let $\lg = \lg_A$ be a nilpotent almost abelian Lie algebra of dimension $2n+2$. Then $\lg$ admits a complex structure if and only if there exists  a partition
  \begin{equation}\label{eq: complex jordan B} n = \sum_{i>0} q_i \cdot i , \qquad q_i\geq 0,\end{equation}
and an index $j>0$ such that the Jordan partition of $A$ is
 \begin{align}\label{eq: jordan A}
 \begin{split}
     2n+1  &= {\sum_{i>0}} m_i \cdot i \\
     &=
     \begin{cases}     
    \displaystyle{\sum_{\stackrel{i>0}{  i\neq j,\, j-1}}} 2q_i \cdot i + (2 q_j +1)\cdot j + (2 q_{j-1}-1)\cdot (j-1), \; &\text{if } j\in I,\\ 
   &  \\
   \hspace*{.25cm} \displaystyle{ \sum_{i>1}} \; 2q_i \cdot i + (2 q_1 +1)\cdot 1 , \; &\text{if } j=1,
 \end{cases}
 \end{split}
 \end{align}
 where $I= \{i>1 \,|\, q_{i-1}>0\}$. 
 
 In particular, each partition of $n$ as in \eqref{eq: complex jordan B} gives rise to $|I|+1$ isomorphism classes of $(2n+2)$-dimensional almost abelian Lie algebras admitting complex structures. 
\end{thm}
 We call the index $j$ in \eqref{eq: jordan A} the size of the overlapping block.
  The reason for choosing this name will become clear in the sequel.

\begin{rem}
 Notice that following the same notation as that of Theorem \ref{thm: ABDGH}, the step of nilpotency of $\lg$ is given by $\max\{ \, i \,\,| \,\, m_i \neq 0\}=\max\{\,j\,,\,\max\{\, i \, | \, q_i\neq0\}\,\}$.    
\end{rem}

  \begin{rem} 
   The partition $n = n\cdot 1$ plays a special role: there are two possible corresponding  partitions of $A$, namely, $2n+1 = (2n+1) \cdot 1$ and $2n+1 = 2 + (2n-1) \cdot 1$. The first one corresponds to the abelian Lie algebra, which we exclude from now on. The second one is $\lg \isom \gothh_3\oplus \IR^{2n-1}$ which sometimes behaves differently, see Remark~\ref{rem: fibre bundle}.
  \end{rem}

Consider two Jordan partitions as in Theorem \ref{thm: ABDGH} and let $j$ be the size of the overlapping block. We define the following matrices: let $B\in \mathsf{M}_{n}(\mathbb{Z})$ be a lower triangular nilpotent matrix  with Jordan partition \eqref{eq: complex jordan B} and starting with a Jordan block of size $j-1$ if $j>1$; if $j=1$,  the order of Jordan blocks can be chosen arbitrarily. 
Set 
 \begin{equation}\label{eq: epsilon}     
  \epsilon = 
  \begin{cases} 1, & \text{if } j>1,\\ 0, & \text{if }j =1,
\end{cases} \end{equation}
and let 
\begin{equation}\label{eq: A}
  A =\begin{pmatrix}
         0 &  0  \\
         v& \tilde B
    \end{pmatrix}   
\in \mathsf{M}_{2n+1}(\mathbb{R}), \quad \text{where } v= \begin{pmatrix}
    \epsilon\\
    0\\
    \vdots \\
    0
\end{pmatrix}, \quad \tilde  B=\begin{pmatrix}
    B&  \\    
    &  B
   \end{pmatrix}.
\end{equation}
Note that $A$ is a matrix in Jordan normal form with Jordan partition \eqref{eq: jordan A} such that the first Jordan block has size $j$. 

We also define
\begin{equation} \label{eq: J_0} 
J_0= \left( \begin{array}{cclc}
       0&-1& &\\
       1& 0& & \\
        & &0 & -I_n \\
        & & I_n & 0
    \end{array}\right) \in \mathsf{M}_{2n+2}(\mathbb{R}), 
\end{equation}
which satisfies $J_0^2 = - \id$. 

\begin{prop}\label{prop: classification}
 Let $\lg$ be a nilpotent almost abelian Lie algebra with complex structure $J$ and let $j$ be the size of the overlapping block defined in Theorem \ref{thm: ABDGH}.
 
 Then there exists a basis  $e_0, \dots, e_{2n+1}$ for $\lg$ with respect to which  the codimension one  abelian ideal is  $\gotha = \langle e_1, \dots, e_{2n+1}\rangle $, the adjoint action  $\phi = \ad_{e_0}|_\gotha$ is represented by the matrix
 $A$ in \eqref{eq: A} with 
 $\epsilon$ as in \eqref{eq: epsilon} and the complex structure is represented by the matrix $J_0$ in \eqref{eq: J_0}.
In particular, if a nilpotent almost abelian Lie algebra admits a complex structure, then the complex structure is unique up to Lie algebra isomorphism.
\end{prop}
\begin{proof}
Note that the adjoint action $\ad_{e_0}$ depends only on the class of $e_0$ modulo $\gotha$, so $\phi$ can only change by a non-zero multiple if we pick a different vector $e_0\in \lg \setminus \gotha$. 

Let $\gothb = \gotha \cap J\gotha$  and pick an arbitrary vector $u \in \gotha \setminus \gothb$. 
We recall from \cite[\S6]{LRV} that  the integrability of the complex structure $J$ implies that $\phi$ preserves $\gothb$ and its restriction $\phi|_\gothb$ commutes with $J|_\gothb$.
  So we can and will consider $\phi|_\gothb$ as a $\IC$-linear transformation of the complex vector space $(\gothb, J|_\gothb)$.

 Set $v = \phi(u) $ and note that $v\in \gothb$. By the discussion in Section \ref{sect: Jordan} we can write 
  \[ \phi(u) =v = w + \phi(w')\]
  such that $w$ can be chosen as a starting vector of a Jordan basis for $\phi|_\gothb$ if $w\neq 0$.
  
We define $e_1 = u-w'$, so that $\phi(e_1) = w$ and let $e_0 = -Je_1$.
Now we distinguish two cases:
 \begin{prooflist}
  \item[Case 1: $\phi(e_1)=w\neq 0$]
    By the discussion in Section \ref{sect: Jordan} we can 
    pick a complex Jordan basis 
    \[ v_1 = \phi(e_1) ,\; v_2, \dots, v_n\]
   with respect to which $\phi|_\gothb$ is represented by a Jordan block matrix $B'$.  

 Now consider the vectors 
 \[e_2 = v_1, \dots, e_{n+1} = v_n, \; e_{n+2} = iv_1, \dots,e_{2n+1} = i v_n \in \gothb.\]
 These form a real basis and if we consider $\phi|_\gothb$ as an $\IR$-linear map, the representing matrix with respect to this basis is
 \[ \mat{  B'& 0  \\
   0  & B'}.
     \]
Since $\phi(e_1)= e_2$ by definition, the map $\phi$ is represented with respect to the chosen basis by the Jordan block matrix $A$ with $\epsilon=1$ , where $B$ is replaced by $B'$. Comparing the Jordan partitions of  $A$, of $B$ and of $B'$,  we have $B = B'$ up to reordering of the Jordan blocks except the first one, as claimed.  
By definition the complex structure acts by the matrix $J_0$.

\item[Case 2: $\phi(e_1) =0$]
As in the previous case,  if $v_1, \dots, v_n$ is a complex Jordan basis, then 
\[ e_0, \;  e_1, \; e_2 = v_1, \dots, e_{n+1} = v_n,\;  e_{n+2} = iv_1, \dots,e_{2n+1} = i v_n\]
is a basis with the required properties. \qedhere
 \end{prooflist}
\end{proof}

\begin{rem}\label{rem: holomorphic fibre bundle}
    As stated in Remark \ref{rem: fibre bundle}, the codimension one abelian ideal of a nilpotent almost abelian Lie algebra $\lg \not\cong \lh_3\oplus \IR^{2n-1}$ is always rational. By a construction similar to \cite[Example~1.14]{rollenske09} one shows that this does not hold for the codimension two ideal $\gothb$ defined in the proof above. Nevertheless, we will see in the next section that almost abelian complex nilmanifolds have the structure of an  iterated principal holomorphic torus bundle.
\end{rem}

It is often useful to encode both the Lie algebra and the complex structure in the differential of the forms of type $(1,0)$.
\begin{cor}\label{cor: structure eq}
 Let $(\lg_A, J)$ be a nilpotent almost abelian Lie algebra with complex structure and with $\epsilon$ as in~\eqref{eq: A}. Let $j$ be the size of the overlapping block in the sense of Theorem \ref{thm: ABDGH} and write the partition \eqref{eq: complex jordan B} as
 \[ n = k_0 + k_1 + \dots + k_r , \qquad  k_i >0, \;\;  i= 1, \ldots , r, \]
 with $k_0 = j-1\geq 0$. 
 \begin{enumerate}
     \item If $\varepsilon=0$ (equivalently $j=1$), then there exists a basis
 \[\alpha,\beta_1^1,\dots,\beta_{k_1}^1,\dots, \beta_1^r, \dots, \beta_{k_r}^r\]
 of $\einsnull{\lg^*}$ such that  $d\alpha =0$ and
 \begin{equation*}
     d\beta^\ell_i =\begin{cases}
            0, & \text{for $i=1$},\\
         (\alpha + \overline\alpha) \wedge \beta^\ell _{i-1} \qquad &\text{for $i>1$}.
     \end{cases}
 \end{equation*}
\item If $\varepsilon =1$, then there exists a basis 
 \[\alpha,\beta_1^0, \dots,\beta_{k_0}^0, \dots, \beta_1^r, \dots, \beta_{k_r}^r\]
 of $\einsnull{\lg^*}$
such that $d\alpha  = 0$ and 
 \begin{equation*} 
d\beta^\ell_i  = \begin{cases}
              0, &  \text{for $\ell>0$ and $i=1$},\\
              (\alpha + \overline\alpha) \wedge \beta^\ell _{i-1} \qquad &\text{for $i>1$},\\
              \alpha\wedge \overline \alpha , & \text{for $\ell = 0$ and $i=1$}.
             \end{cases}\notag
\end{equation*}
\end{enumerate}
\end{cor}
\begin{proof}
Let $e_0,\dots,e_{2n+1}$ be a basis for $\lg$ as constructed in the proof of Proposition \ref{prop: classification}. The bracket relations of this basis are expressed by the matrix $A$ given in~\eqref{eq: A}.
Starting with the case $\varepsilon =0$ the only non-zero brackets are $[e_0,e_{k-1}]= \mu_k e_{k}$, where $\mu_k\in \{0,1\}$ and $\mu_k=0$ precisely if the $k$-th row of $A$ corresponds to the first row of a Jordan block in $A$, in particular $[e_0,e_1]=0$. Hence, for the dual basis $e^0,\dots,e^{2n+1}$ of $e_0,\dots,e_{2n+1}$ we get $de^{k} = -\mu_k e^0 \wedge e^{k-1}$ and since the matrix $\tilde{B}$ in \eqref{eq: A} consists of two identical blocks we also have $\mu_k = \mu_{k+n}$ for $k \geq 2$. With respect to the complex structure $J_0$ a basis of $\einsnull{\lg^*}$ is now given by the forms
\[e^0+ie^1, e^2+ie^{n+2}, \dots , e^{n+1}+ie^{2n+1}.\]
We define $\alpha$ to be the closed form $\alpha = -\tfrac{1}{2}(e^0+ie^1)$ and compute
\[
d(e^k +ie^{k+n}) = -\mu_k(e^0 \wedge e^{k-1} +ie^0 \wedge e^{k+n-1}) = \mu_k(\alpha + \overline{\alpha}) \wedge (e^{k-1}+ie^{k+n-1}).
\]
Now, we set the forms $e^k +ie^{k+n}$ with $\mu_k=0$ corresponding to the first row of the $\ell$-th Jordan block of $B$ in~\eqref{eq: A} as $\beta_1^\ell$ and continue this with $e^{k+m} +ie^{k+n+m} =\beta_{m+1}^\ell$ for every $m$ smaller then the size of the Jordan block. Repeating this for every Jordan block of $B$ we obtain the desired structure equations.

The case $\varepsilon=1$ is done in the same fashion, except there is an additional non-zero bracket given by $[e_0,e_1] =e_2$. If $\alpha = e^0+ie^1$, we get $d(e^2+ ie^{2+n}) = \frac{1}{2i} \alpha \wedge \overline{\alpha}$. Setting $\beta_1^0 =2i(e^2+ ie^{2+n})$ and repeating the same argument as in the previous case we get the structure equations given above.
\end{proof}

\begin{rem}\label{rem: structure eq}
 From the structure equations in Corollary \ref{cor: structure eq} we see that a complex structure on $\lg_A$ is always nilpotent in the sense of \cite{Nilpotent} and never complex parallelisable. It is abelian if and only if the commutator is of dimension one, that is $\lg_A$ is isomorphic to a direct sum of $\lh_3$ and an abelian Lie algebra.
\end{rem}

\subsection{Proof of Theorem \ref{thm: stbs} $(i)$ and $(ii)$}
We start with a consequence of the classification above.
\begin{cor}\label{cor: stbs}
Let $(\lg_A, J)$ be a nilpotent almost abelian Lie algebra with complex structure, Jordan partitions as in Theorem \ref{thm: ABDGH}, and size of the overlapping block $j$. Let $\nu$ be the size of the largest Jordan block, that is, the nilpotency index of $\lg_A$.
 
 Then the filtration 
  \[ 0 = \kz^0\lg \subset \kz^1\lg \subset \dots \subset \kz^{j-1}\lg \subset \kz^{j-1}\lg +\kc^{\nu-j}\lg \subset \dots \subset \kz^{j-1}\lg +\kc^1\lg \subset \lg\]
  is a stable principal torus bundle series for $\lg = \lg_A$. 
 \end{cor} 
\begin{proof}
The filtration is defined over $\IQ$ with respect to any rational structure, because it is defined in terms of the Lie bracket, see \cite[p. 208]{Cor-Green}.

It remains to show $J$-invariance for every complex structure $J$. Since every Lie algebra isomorphism preserves the above filtration it follows from Proposition~\ref{prop: classification} that we can assume $J= J_0$. Choose a basis $e_0,\dots,e_{2n+1}$ of $\lg$ as constructed in Proposition \ref{prop: classification}.

First, we consider the case $\epsilon = 0$, that is, we have $j=1$. The descending central series of $\lg$ is given by the matrix $A$ in (\ref{eq: A}), where $e_i$ is an element of $\kc^k\lg$ if and only if all rows from the $(i-k+1)$-th row to the $i$-th row are non-zero. In particular, the vector $e_2$ is not in the commutator ideal for $\varepsilon =0$. Suppose $e_i$ is an element of $\kc^k\lg$, then so is $e_{i+n}$ and since $J(e_i) = e_{i+n}$  for $J=J_0$ and $i=2,\dots,n+1$, the complex structure $J$ has to preserve the descending central series.

Now assume that $j>1$, that is, $\epsilon =1$. In this case $e_2$ is in the commutator of $\lg$ and thus the commutator is no longer $J$-invariant. However, since $e_1$ is not in the centre of $\lg$ if $\varepsilon =1$ and again $J(e_i) = e_{i+n}$  for $i=2,\dots,n+1$, the complex structure preserves the centre $\kz\lg$. We can now consider $\lg/\kz\lg$, which is an almost abelian Lie algebra with complex structure and associated partition $\sum_i m_i\cdot (i-1)$. Repeating the same argument for the Lie algebra $\lg/\kz\lg$ we obtain an almost abelian Lie algebra with complex structure and associated partition $\sum_i m_i\cdot (i-2)$. Proceeding by induction, we reach the case where $j=1$ and we get the claimed $J$-stable filtration.

The last condition of Definition \ref{def: stbs} is clear. 
\end{proof}

Now let $X$ be an almost abelian complex nilmanifold. 
Combining Corollary \ref{cor: stbs} with Proposition \ref{prop: stbs}, we get the desired geometric description of $X$ as an iterated principal holomorphic torus bundle. 

To deduce the description of the Dolbeault cohomology via left-invariant forms, we can apply the results of Console and Fino \cite[Section~5]{con-fin01}, see also \cite[Theorem~2]{rollenske11}. Alternatively, we can use the fact that the complex structure is nilpotent (Remark \ref{rem: structure eq})  and apply the main theorem of \cite{rtx20}. This finishes the proof of Theorem \ref{thm: stbs} $(i)$ and $(ii)$. \qed

\section{Cohomology, Fr\"olicher spectral sequence, and deformations}\label{sect: computation}
Throughout this section we fix an almost abelian complex nilmanifold $X=(M,J)$ and the associated Lie algebra $\lg =\lg_A$ with matrix $A$ as in (\ref{eq: A}). Also, as in the previous section we denote a codimension one abelian ideal of $\lg$ by $\gotha$ and let $\gothb = J\gotha \cap \gotha$.

\subsection{Consequences of the Hochschild-Serre spectral sequence}

The aim is to compute the cohomological invariants of the complex manifold $X=(M,J)$. We will start with general considerations and explain later how to explicitly compute Betti numbers and Hodge numbers in any particular case.

Let us start with the de Rham cohomology and consider the exact sequence 
\[ 0 \to \gotha \to \lg \to \gothh \to 0\]
 so that $\gothh$ is the real 1-dimensional abelian Lie algebra.
 \begin{prop}\label{prop: deRham}
  The (real) de Rham cohomology of the almost abelian nilmanifold $M$ is \[ H^k_\text{dR}(M) = H^k(\lg, \IR) \isom H^0(\lh, \Wedge^k \gotha^*) \oplus H^1(\lh, \Wedge^{k-1}\gotha^*),\]
where the right-hand side is the Lie algebra cohomology for the $\lh$-module structure on $\gotha$ given by the matrix $A$. 
 \end{prop}
 \begin{proof}
 This is either the Leray-Hirsch theorem applied to the fiber bundle $M\to S^1$ explained in Remark \ref{rem: fibre bundle} or the Hochschild-Serre spectral sequence \cite[Section~7.5]{Weibel} in Lie algebra cohomology combined with Nomizu's theorem; see Remark~\ref{rem: nomizu}. Since $\dim \lh= 1$, the Hochschild-Serre spectral sequence degenerates at the second page $E_2^{r,s} = H^r(\lh, H^s(\gotha, \IR))$ and $E_2^{r,s}=0$ for $r\geq 2$. Therefore,
 \[
 H^k(\lg,\IR) \cong H^0(\lh, H^k(\gotha, \IR)) \oplus H^1(\lh, H^{k-1}(\gotha, \IR)),
 \]
 and since the Lie algebra $\gotha$ is abelian we have $H^s(\gotha, \IR)\isom \Wedge^s\gotha^*$.
\end{proof}

Now let us compute the Dolbeault cohomology groups. For this consider the 1-dimensional complex Lie algebra $(\gothf  = \lg/\gothb,J)$. Since $\gothb$ is $J$-invariant, we get an  exact sequence of complex Lie algebras
\begin{equation}\label{eq: mod b} 
0 \to \nulleins\gothb \to \nulleins\lg \to \nulleins\gothf \to 0.
 \end{equation}

As explained in Section \ref{sect: Lie algebra cohomology}, to compute the Lie algebra Dolbeault cohomology we need to consider the $\nulleins{\lg}$-module $\Wedge^p\einsnull{\lg^*}$, where the action is given by 
\begin{equation}\label{eq: action}
 \nulleins{\lg}\times \einsnull{\lg^*} \to \einsnull{\lg^*},\qquad  (\overline X, \gamma) \mapsto-\overline X\intprod \delbar\gamma.
\end{equation}

\begin{prop}\label{prop: Dolbeault}
The Dolbeault cohomology groups of the almost abelian complex nilmanifold $X$ can be computed as
\begin{align*}
 H^{p,q}_{\delbar}(X) & \isom H^q(\nulleins{\lg} , \Wedge^p \einsnull{\lg^*})\\
 & \isom H^0(\nulleins\gothf, \Wedge^q\nulleins{\gothb^*}\tensor \Wedge^p \einsnull{\lg^*})\oplus H^1(\nulleins\gothf, \Wedge^{q-1}\nulleins{\gothb^*}\tensor \Wedge^p \einsnull{\lg^*}). 
 \end{align*}
 \end{prop}
\begin{proof}
By Theorem \ref{thm: stbs} the Dolbeault cohomology of $X$ is computed by left-invariant forms, that is, we have $H^{p,q}_{\delbar}(X) \isom H^q(\nulleins{\lg} , \Wedge^p \einsnull{\lg^*})$. To obtain the second isomorphism we apply the Hochschild-Serre spectral sequence \cite[Section~7.5]{Weibel} to the exact sequence \eqref{eq: mod b} giving us the $E_2$-page as 
 \[ E_2^{r,s} = H^r(\nulleins{\gothf}, H^s(\nulleins{\gothb},  \Wedge^p \einsnull{\lg^*})) \; \implies H^{r+s}(\nulleins\lg , \Wedge^p \einsnull{\lg^*}).\]
Note that by the structure equations given in Corollary \ref{cor: structure eq} the action given by \eqref{eq: action} of $\lg^{0,1}$ restricted to the subalgebra $\nulleins{\gothb}$ on $\Wedge^p\einsnull{\lg^*}$ is trivial and since $\nulleins{\gothb}$ is abelian the Chevalley-Eilenberg complex gives us
\[ H^s(\nulleins{\gothb},  \Wedge^p \einsnull{\lg^*}) = \Wedge^s\nulleins{\gothb^*}\tensor  \Wedge^p \einsnull{\lg^*}.\]

As in the proof of Proposition \ref{prop: deRham} since $\gothf^{0,1}$ is of dimension one, we have $ E_2^{r,s} =0$ for $r\geq 2$, so the spectral sequence has to degenerate at $E_2$. This gives the claimed isomorphism.
\end{proof}
Note that  in general the spectral sequence used in the proof is purely algebraic and not the Leray-Serre spectral sequence of a holomorphic fibration, see Remark~\ref{rem: holomorphic fibre bundle}.

We can now prove the third item of Theorem \ref{thm: stbs}.
\begin{thm}\label{thm: Froelicher}
 Let $X$ be an almost abelian complex nilmanifold. Then we have a (non-canonical) isomorphism
 \[ H^k(X, \IC) \isom \bigoplus_{p+q = k} H_{\delbar}^{p,q}(X).\]
 In particular, the Betti numbers equal the sum of the Hodge numbers and the Fr\"olicher spectral sequence degenerates at the first page.
\end{thm}
\begin{proof}
Starting from Proposition \ref{prop: Dolbeault}, let us consider the total Dolbeault cohomology, disregarding the bidegree decomposition for the moment:
\begin{align*}
 H^{\bullet,\bullet}_{\delbar} (X) & = \bigoplus_{p,q}  H^{p,q}_{\delbar}(X)\\
 & = \bigoplus_{p,q} H^q(\nulleins\lg , \Wedge^p \einsnull{\lg^*})\\
 & \isom  \bigoplus_{p,q,r} H^r(\nulleins\gothf, \Wedge^{q}\nulleins{\gothb^*}\tensor \Wedge^p \einsnull{\lg^*})\\
 & \isom   H^\bullet(\nulleins\gothf, \bigoplus_{p,q}\Wedge^{q}\nulleins{\gothb^*}\tensor \Wedge^p \einsnull{\lg^*})\\
 & \isom   H^\bullet\left(\nulleins\gothf, \Wedge^\bullet\left(\nulleins{\gothb^*}\oplus\einsnull{\lg^*}\right)\right)
\end{align*}

Repeating this for the complexified de Rham cohomology using Proposition \ref{prop: deRham} we get 
\begin{align*}
 H^\bullet(X, \IC) \isom H^\bullet \left( \lh_\IC , \Wedge^\bullet \gotha_\IC^*\right).
\end{align*}
Since both $\nulleins{\gothf} $ and $\lh_\IC$ are one-dimensional complex Lie algebras, it is sufficient to show that there is an isomorphism between them making $\left(\nulleins{\gothb^*}\oplus\einsnull{\lg^*}\right)$ and $\gotha^*_\IC$ into isomorphic modules.
This follows immediately from Lemma \ref{lem: module structures} proved below.

Therefore the total cohomologies of their exterior algebras are isomorphic and we get 
\[ \sum_{k \geq 0} b_k(X) = \sum_{p,q\geq 0} h^{p,q}(X).\]

Thus the Fr\"olicher spectral sequence $E_1^{p,q} = H^{p,q}_{\delbar} (X) \implies H^{p+q}(X, \IC)$ has to degenerate at $E_1$ as claimed. 
\end{proof}
\begin{rem}
 Note that the $E_1$-degeneration of the Fr\"olicher spectral sequence cannot be caused by the $\del\delbar$-Lemma being valid on $X$, because this would imply formality and complex nilmanifolds are never formal \cite{hasegawa1989minimal}. In particular, we should not expect that our Hodge numbers are symmetric under conjugation or that we have a canonical decomposition.
 
 This becomes explicit in Example \ref{exam: explicit computation}.
\end{rem}

\subsection{Proof of Theorem \ref{thm: deform}}

For the first item \cite[Thm. 2.6]{rollenske09b} implies that small deformations of $X$ are again complex nilmanifolds. The unobstructedness of deformations holds in general for complex manifolds with trivial canonical bundle and with Fr\"olicher spectral sequence degenerating at $E_1$, see \cite[Theorem~3.3]{ACRT18} and references therein. The triviality of the canonical bundle was observed in \cite{salamon01} and the degeneration of the spectral sequence is Theorem \ref{thm: Froelicher} above. By \cite[Theorem~4.2]{wavrik1969obstructions} the Kuranishi family of $X$ is universal if and only if the dimension of the automorphism group $\Aut(X_t)$ of the complex manifolds $X_t$ parametrized by the Kuranishi space is constant, provided the Kuranishi space is reduced. In our case it is even smooth since $X$ is unobstructed. Let $\lg$ be the Lie algebra associated to $X$. Then every complex manifold $X_t= (M,J_t)$ is again a complex nilmanifold with Lie algebra $\lg$ and as such their canonical bundle is trivial. Therefore, we have $\dim \Aut(X_t) = \dim H_{\delbar}^{n-1,0}(X_t) = \dim H^{n-1,0}(\lg, J_t)$. Indeed, the first equality comes from the fact that the Lie algebra of $\Aut(X_t)$ is the space of global holomorphic vector fields on $X_t$ combined with Serre duality and the second equality is Theorem~\ref{thm: stbs}. Since the complex structure on $\lg$ is unique (Proposition \ref{prop: classification}), the dimension of the automorphism group does not change in the family. 
 
For the last item we apply \cite[Theorem A]{rollenske09}. The last subspace $\kz^{j-1}\lg + \kc^1\lg$ of the stable principal torus bundle series from Corollary \ref{cor: stbs} is an abelian ideal in $\lg$ and complex tori are a good fibre class in the sense of \cite{rollenske09} by \cite[Remark~2.12]{rollenske09}.
This concludes the proof. \qed

\subsection{Explicit computations}
We will now show how to explicitly compute Hodge and Betti numbers. This is straightforward in principle, but will become messy when the Jordan partition \eqref{eq: jordan A}  has many summands.

As a preparation consider a $1$-dimensional complex Lie-algebra $\gothq$ and the  $\gothq$-module $W_i=\IC^i$ where the action is given by the Jordan block matrix $\mathcal J_i$. More precisely, we have the action
\[
\gothq \times W_i \to W_i, \; (\lambda, w) \mapsto \lambda \mathcal J_i w.
\] 
\begin{lem}\label{lem: cohomology}
 We have 
 \[ \dim H^0(\gothq, W_i)  = \dim H^1(\gothq, W_i) = 1.\]
\end{lem}
\begin{proof}
 The Chevalley-Eilenberg complex for $W_i$ can be identified as 
 \[ \Wedge^0 \gothq^* \tensor W_i \isom \IC^i \overset{\cdot \mathcal J_i}{\longrightarrow} \IC^i \isom \gothq^*\tensor W_i,\]
 which gives the result.
\end{proof}

\begin{lem}\label{lem: module structures}
 Let $\lg =\lg_A$ be a nilpotent almost abelian Lie algebra with complex structure. Let \eqref{eq: jordan A}  
 and \eqref{eq: complex jordan B} be the associated  Jordan partitions from Theorem \ref{thm: ABDGH}. Let $j$ be the size of the overlapping block.
 Then, with $q_i$ as in \eqref{eq: complex jordan B} and $m_i$ as in \eqref{eq: jordan A},
 \begin{enumerate}
  \item As a $(\lg /\gotha)_\IC$-module we have 
  \[ \gotha_\IC^* \isom \bigoplus_{i>0} m_i W_i.\]
  \item As a $\nulleins{\gothf}$-module we have 
  \[ \nulleins{\gothb^*} \isom  \bigoplus_{i>0} q_i W_i.\]
  \item As a $\nulleins{\gothf}$-module we have
  \[ \einsnull{\lg^*}\isom
  \begin{cases}
(q_j+1)W_{j} \oplus (q_{j-1} -1) W_{j-1}\oplus  \bigoplus_{i>0, i\neq j, j-1} q_i W_i, & j>1\\
(q_1+1)W_{1} \oplus  \bigoplus_{i>1} q_i W_i, & j=1
  \end{cases}
.\]
 \end{enumerate}
\end{lem}
\begin{proof}
Choose a basis $e_0,\dots,e_{2n+1}$ of $\lg$ as constructed in Proposition \ref{prop: classification} with dual basis $e^0,\dots,e^{2n+1}$. For every Jordan block $\mathcal
J_i$ in $A$ (\ref{eq: A}) we have a subspace $W_i$ of $\gotha_\IC^*$ spanned by the vectors $e^k$ such that the $k$-th column of $A$ corresponds the Jordan block $\mathcal J_i$. Clearly, these subspaces are invariant under the action of $A$ on $\gotha_\IC^*$. This proves the first item.

For the second and third item we have to consider the action given in (\ref{eq: action}). Choose a basis $\alpha, \beta^0_1, \dots, \beta_{j-1}^0, \beta_1^1, \dots, \beta_{k_1}^1, \dots,\beta^r_{k_r}$ of the space $\einsnull{\lg^*}$ with structure equations as in Corollary \ref{cor: structure eq}. Then $\nulleins{\gothf}$ is generated by the class of $\overline{X}$ dual to $\overline\alpha$ and  $\nulleins{\gothb^*} = \langle \overline{\beta^0_1}, \dots, \overline{\beta_{j-1}^0}, \overline{\beta_1^1},\dots,\overline{\beta_{k_1}^1},\dots,\overline{\beta^r_{k_r}}\rangle$. Starting  with the second item, the action of $\nulleins{\gothf}$ on $\nulleins{\gothb^*}$ is now given by
\[(\overline{X}, \overline{\beta_i^\ell}) \mapsto
\begin{cases}
    0 & \text{for $i=1$},\\
   - \overline{\beta_{i-1}^\ell} & \text{for $i>1$},
\end{cases}
\]
since $\overline{X} \intprod \delbar \overline{\beta_i^\ell} = \overline{X} \intprod( \overline{\alpha} \wedge \overline{\beta_{i-1}^\ell}) =\overline{\beta_{i-1}^\ell}$. Therefore, the subspaces $W_{k_i} =\langle\overline{\beta_1^i},\dots,\overline{\beta_{k_i}^i}\rangle$ of $\nulleins{\gothb^*}$ are invariant under the action of $\nulleins{\gothf}$. By the same computation the action of $\nulleins{\gothf}$ on $\einsnull{\lg^*}$ is given by $(\overline X, \alpha) \mapsto 0$ and
\[(\overline{X}, {\beta_i^\ell}) \mapsto
\begin{cases}
    0 & \text{for $\ell>0$ and $i=1$},\\
    {-\beta_{i-1}^\ell} & \text{for $i>1$},\\
    -\alpha & \text{for $\ell =0$ and $i=1$},
\end{cases}
\]
which gives the decomposition in the last item.
\end{proof}

In principle, the preceding lemmas allow for the computation of Hodge and Betti numbers in every case: for a nilpotent representation $V$ of a $1$-dimensional complex Lie algebra $\IC f$,  $V$ becomes an $\liesl_2(\IC)$-representation (see section \ref{sect: Jordan}) and decomposes as $V= \oplus _{i>0}\, m_i W_i$, where $W_i$ is the irreducible $\liesl_2(\IC)$-representation of dimension $i$. 
Denote by $\delta ( V )$ the number of $\liesl_2(\IC)$-irreducible summands  in a decomposition of $V$, that is, the number of Jordan blocks of a representing matrix $A$ of $f$. In other words:
\begin{equation}\label{eq:delta}
    \delta (V)= \sum _{i>0} m_i, \qquad \text{where } m_i=\dim \operatorname{Hom}_{\liesl_2(\IC)} (W_i, V).
\end{equation}
 
 \begin{prop}\label{prop: counting}
 Consider two irreducible $\liesl_2(\IC)$-representations $W_i$ and $W_k$.
 \begin{enumerate}
  \item For the tensor product we have $\delta (W_i \tensor W_k)  = \min\{i, k\}$.
 \item For the exterior powers we get the following:
 \begin{enumerate}
 \item We have $\delta (\Wedge ^r W_i) = \delta (\Wedge ^{i-r} W_i)$, so it is enough to know these values for $0\leq r \leq \left\lfloor\frac i2 \right\rfloor$.
 \item Following the notation of \cite{Almkvist} we define $A(m,n,r)$ as the number of partitions of $m$ with at most $n$ summands all of size $\leq r$. Then 
     \[
        \delta  (\Wedge^r W_i)  = A\left(\left\lfloor \frac{r(i-r)}{2} \right\rfloor, i-r,r\right).
     \] 
In particular, we have  $\delta ( \Wedge^0 W_i ) = \delta(\Wedge^1 W_i) = 1$ and $\delta(\Wedge^2 W_i) = \left\lfloor \frac{i}{2} \right\rfloor$.
 \end{enumerate}
  \end{enumerate}
\end{prop}
\begin{proof}
 For the first item  this is the Clebsch-Gordan rule \cite[Exercise~22.7]{Humphreys} for the tensor product of $\liesl_2(\IC)$-representations, which gives for $i\leq k$
 \[\delta(W_i \tensor W_k) = \delta\left( \bigoplus_{n=0}^{i-1} W_{k+i-1-2n}\right) = i.\]

 For part $(a)$ of the second item note that $\Wedge ^{i-k} W_i \isom \left(\Wedge ^{k} W_i\right)^* \tensor \det W_i \isom\Wedge^k W_i$ as $\liesl_2 (\IC )$-representations. Finally, by \cite[Proposition~1.21]{Almkvist} we have 
\begin{align*}
    \delta  (\Wedge^r W_i)  &= \delta \left( \bigoplus^{r(i-r)}_{n=1}\left( A\left(\tfrac{r(i-r)-n+1}{2},i-r,r\right) -A\left(\tfrac{r(i-r)-n-1}{2},i-r,r\right) \right)W_n  \right)\\
    &= \sum^{r(i-r)}_{n=1}  A\left(\tfrac{r(i-r)-n+1}{2},i-r,r\right) -A\left(\tfrac{r(i-r)-n-1}{2},i-r,r\right)\\
    &= A\left(\left\lfloor \frac{r(i-r)}{2} \right\rfloor, i-r,r\right)
\end{align*}
This proves part $(b)$.
\end{proof}

\begin{prop}\label{prop: betti-hodge}
 Let $X$ be an almost abelian complex nilmanifold. In the notation introduced above, the cohomological invariants of $X$ are
 \begin{align*}
 b_k(X)& = \delta(\Wedge^k \gotha^*) + \delta(\Wedge^{k-1} \gotha^*)\\
 h^{p,q}(X) &  = \delta(\Wedge^q\nulleins{\gothb^*}\tensor \Wedge^p \einsnull{\lg^*}) + \delta( \Wedge^{q-1}\nulleins{\gothb^*}\tensor \Wedge^p \einsnull{\lg^*}). 
\end{align*}
\end{prop}
\begin{proof}
This follows immediately from Proposition \ref{prop: deRham}, Proposition \ref{prop: Dolbeault}, Lemma \ref{lem: cohomology}, and Lemma \ref{lem: module structures}.
\end{proof}
\begin{cor}\label{cor: hodge symmetric}
    Let $X$ be an almost abelian complex nilmanifold with $\epsilon$ as in \eqref{eq: epsilon}.
    \begin{enumerate}
        \item If $\epsilon=0$, then the Hodge numbers of $X$ satisfy $h^{p,q}(X) = h^{q,p}(X)$ and thus the odd Betti numbers of $X$ are even.
        \item  If $\epsilon =1$, then $b_1(X)$ is odd and Hodge symmetry does not hold.
    \end{enumerate}
\end{cor}
\begin{proof}
For the first item it follows from Lemma \ref{lem: module structures} that in the case $\varepsilon=0$ or equivalently $j=1$ we have
\[ \nulleins{\gothb^*} \isom  \bigoplus_{i>0} q_i W_i, \qquad  \einsnull{\lg^*}\isom (q_1+1)W_{1} \oplus  \bigoplus_{i>1} q_i W_i\isom W_1 \oplus \nulleins{\gothb^*}.\]
Therefore
\[ \Wedge^p \einsnull{\lg^*} \isom \Wedge^p \nulleins{\gothb^*} \oplus \Wedge^{p-1} \nulleins{\gothb^*}. \]
According to Proposition \ref{prop: betti-hodge} we have
\begin{align*} 
h^{p,q}(X) & =\delta\left(\Wedge^q \nulleins{\gothb^*} \tensor \left(\Wedge^p \nulleins{\gothb^*} \oplus \Wedge^{p-1} \nulleins{\gothb^*}\right)\right) \\ 
& \quad +\delta\left(\Wedge^{q-1} \nulleins{\gothb^*} \tensor \left(\Wedge^p \nulleins{\gothb^*} \oplus \Wedge^{p-1} \nulleins{\gothb^*}\right)\right),
\end{align*}
and this expression is symmetric with respect to $p$ and $q$.\\
If $\epsilon =1$, then Proposition \ref{prop: betti-hodge} implies
\[
b_1(X) =  \delta(\gotha^*) +\delta \left( \Wedge^0\gotha^*\right) = 2 \delta \left(\einsnull{\gothb^*}\right) +1.
\]
This proves the second item.
\end{proof}
\begin{exam}\label{exam: explicit computation}
We compute some Hodge and Betti numbers in the case where $\dim \lg = 2n+2 \geq 6$,  the Jordan partition of $B$ is $n = 1 \cdot n$ and $\epsilon =1$, so we have $\gotha^* \isom W_{n+1} \oplus W_n$. 
The first Betti numbers are: 
 \begin{align*}
 b_1(X) &  = \delta(\gotha^*) + \delta (\IC) = 2+1 = 3,\\
 b_2(X) &= \delta(\Wedge ^2\gotha^*) + \delta(\gotha^*),\\
 & = \left(\delta(\Wedge ^2W_n) + \delta (W_n\tensor W_{n+1})  + \delta (\Wedge^2 W_{n+1})\right)   + \delta(\gotha^*) \\
 & =  \left\lfloor \frac{n}{2} \right\rfloor + n + \left\lfloor \frac{n+1}{2} \right\rfloor  +2 = 2n +2.
 \end{align*}
The first Hodge numbers can be computed using Lemma \ref{lem: module structures}: 
\begin{align*}
 h^{1,0} (X) & = \delta (\einsnull{\lg^*}) =1,\\
 h^{0,1} (X) &= \delta (\nulleins{\gothb ^*})+ \delta(\IC) =2,\\
 h^{2,0} (X) &= \delta (\Wedge ^2 \einsnull{\lg^*})  = \delta (\Wedge ^2 W_{n+1}) = \left\lfloor \frac{n+1}{2} \right\rfloor,\\
 h^{0,2} (X) &= \delta (\Wedge ^2 \nulleins{\gothb^*}) + \delta (\nulleins{\gothb^*})  = \delta (\Wedge ^2 W_{n}) +1 = \left\lfloor \frac{n}{2} \right\rfloor +1,\\
 h^{1,1} (X) &= \delta (\nulleins{\gothb^*}\tensor \einsnull{\lg^*}) + \delta (\nulleins{\lg^*})  = n +1
\end{align*}
and we see that $b_2(X) =  h^{0,2} (X)+ h^{1,1} (X) + h^{2,0} (X) $ as predicted by Theorem \ref{thm: Froelicher}.

Most other Betti and Hodge numbers will involve a count of partitions as in Proposition \ref{prop: counting} $(ii)\, (b)$, so we stop the analysis here. 

\end{exam}

\subsection{The 2-step nilpotent case} \label{subsec: 2-step}
In this subsection we will compute some Betti and Hodge numbers of certain almost abelian complex nilmanifolds where $\gothg$ is 2-step nilpotent. First, we state some results that we will need for the computations.  

\begin{lem}
The $\liesl_2(\IC)$-modules $\Wedge^k nW_2$, for $k=1,\ldots, 5$, $n\in \IN$, decompose as a sum of irreducible submodules as follows\footnote{As usual, $\tbinom{p}{q}=0$ if $p<q$.}:
\begin{itemize}
    \item $\Wedge^1 nW_2=nW_2$,
    \smallskip
    \item $\Wedge^2 nW_2 =\binom{n+1}{2}W_1\oplus \binom{n}{2} W_3$,
    \smallskip
    \item $\Wedge^3 nW_2 =2\binom{n+1}{3}W_2\oplus \binom{n}{3} W_4$,
    \smallskip
    \item $\Wedge^4 nW_2 =\frac{n}{2}\binom{n+1}{3}W_1\oplus 3\binom{n+1}{4} W_3\oplus \binom{n}{4}W_5$,
    \smallskip
    \item $\Wedge^5 nW_2 =n\binom{n+1}{4}W_2\oplus 4\binom{n+1}{5} W_4\oplus \binom{n}{5}W_6$.
\end{itemize}
\end{lem}

\begin{proof}
The proof follows recursively from the relation
\[ \Wedge^{k} nW_2= \Wedge^{k} (n-1)W_2 \oplus (\Wedge^{k-1} (n-1)W_2) \otimes W_2  \oplus \Wedge^{k-2} (n-1)W_2, \]
together with some lengthy combinatorial computations. 

For example, let us compute the case $k=2$. The formula above in this case becomes
\[ \Wedge^{2} nW_2= \Wedge^{2} (n-1)W_2 \oplus (n-1)W_2 \otimes W_2  \oplus W_1. \]
The well-known Clebsch-Gordan formula implies that $W_2\tensor W_2=W_1\oplus W_3$, so that 
\begin{equation}\label{eq: wedge-n}
   \Wedge^{2} nW_2 = \Wedge^{2} (n-1)W_2 \oplus n W_1 \otimes (n-1) W_3. 
\end{equation}
From this equation, together with $\Wedge^{2} W_2=W_1$, it follows that the multiplicity of $W_i$ is $0$ for $i\neq 1$ and $i\neq 3$. Moreover, if $m^i_n$ denotes the multiplicity of $W_i$ in $\Wedge^{2} nW_2$ ($i=1$ or $i=3$) then \eqref{eq: wedge-n} gives
\[ m^1_n=m^1_{n-1}+n, \qquad m^3_n=m^3_{n-1}+(n-1).\]
We have the initial values $m^1_1=1$, $m^3_1=0$. It is easily verified that the solutions to these recursive equations are 
\[ m^1_n=\tbinom{n+1}{2}, \qquad m^3_n=\tbinom{n}{2},\]
and the proof for $k=2$ is complete.
\end{proof}

Using the fact that $\delta$ is distributive with respect to the direct sum of representations, we obtain the following values:

\begin{cor}\label{cor: delta}
For $j=1,\ldots,5$ we have that
\begin{itemize}
    \item $\delta(\Wedge^1 nW_2)=n$,
    \smallskip
    \item $\delta(\Wedge^2 nW_2)=n^2$,
    \smallskip 
    \item $\delta(\Wedge^3 nW_2)=n\binom{n}{2}$,
    \smallskip
    \item $\delta(\Wedge^4 nW_2)=\binom{n}{2}^2$,
    \smallskip
    \item $\delta(\Wedge^5 nW_2)=\binom{n}{2}\binom{n}{3}$.
\end{itemize}
\end{cor}

We are now ready to begin our computations. In the following, we calculate some Betti and Hodge numbers for two classes of $2$-step nilpotent almost abelian Lie algebras.

\begin{exam}\label{exam: 1}
 Assume first that the Jordan partition of $B$ is $n= m\cdot 2$ (that is, $q_2=m$ for some $m\geq 1$) and $\epsilon =0$, so that $\gotha^* \cong W_1 \oplus n W_2$. Hence $\Wedge^k \gotha^*= \Wedge^k nW_2 \oplus \Wedge^{k-1} nW_2$. Using Proposition \ref{prop: betti-hodge} and Corollary \ref{cor: delta} we compute the Betti numbers $b_k(X)$, $k\leq 5$:
\begin{align*}
    b_1(X)&=\delta(\gotha^*)+\delta (\IC) =n+2,  \\ b_2(X)&=\delta(\Wedge^2\gotha^*)+\delta(\gotha^*)=(n^2+n)+(n+1)  =(n+1)^2,   \nonumber \\    b_3(X)&=\delta(\Wedge^3\gotha^*)+\delta(\Wedge^2\gotha^*)=\left(n\binom{n}{2}+n^2\right)+(n^2+n)=3 \binom{n+2}{3}, \nonumber  \\
    b_4(X)&= \delta(\Wedge^4\gotha^*)+\delta(\Wedge^3\gotha^*)=\left(\binom{n}{2}^2+n\binom{n}{2}\right)+\left(n\binom{n}{2}+n^2\right)=\binom{n+1}{2}^2 , \nonumber  \\
    b_5(X)&= \delta(\Wedge^5\gotha^*)+\delta(\Wedge^4\gotha^*)=\left( \binom{n}{2}\binom{n}{3}+\binom{n}{2}^2\right)+\left(\binom{n}{2}^2+n\binom{n}{2}\right)=\binom{n}{2}\binom{n+2}{3}. \nonumber
\end{align*}

For the computation of the first Hodge numbers, we note that Lemma \ref{lem: module structures} gives isomorphisms $\einsnull{\lg^*}\cong W_1\oplus mW_2$ and $\nulleins{\gothb^*}\cong mW_2$, while Corollary \ref{cor: hodge symmetric} gives $h^{p,q}(X)=h^{q,p}(X)$. It follows from Proposition \ref{prop: betti-hodge} (see also the proof of Corollary \ref{cor: hodge symmetric}) that
\begin{align*} 
h^{p,q}(X) & =\delta(\Wedge^q mW_2\tensor (\Wedge^p mW_2\oplus \Wedge^{p-1} mW_2)) \\ & \qquad +\delta(\Wedge^{q-1} mW_2\tensor (\Wedge^p mW_2\oplus \Wedge^{p-1} mW_2)).
\end{align*}
For small values of $p,q$ we have 
\begin{align*}
    h^{1,0}(X)&=\delta(W_1\oplus mW_2)= m+1,  \\ 
    h^{2,0} (X)& =\delta(\Wedge^2mW_2\oplus mW_2)= m^2+m, \\    
    h^{3,0}(X)&  = \delta(\Wedge^3 mW_2\oplus \Wedge^2 mW_2)=m\tbinom{m}{2} + m^2=m\tbinom{m+1}{2},  \\
    h^{4,0}(X)&= \delta(\Wedge^4 mW_2\oplus \Wedge^3 mW_2)=\tbinom{m}{2}^2+m\tbinom{m}{2}=\tbinom{m}{2}\tbinom{m+1}{2},   \\
    h^{5,0}(X)&= \delta(\Wedge^5 mW_2\oplus \Wedge^4 mW_2)=\tbinom{m}{2}\tbinom{m}{3}+ \tbinom{m}{2}^2= \tbinom{m}{2}\tbinom{m+1}{3},\\
    h^{1,1}(X) & = \delta(mW_2\tensor ( mW_2\oplus W_1)) +\delta(mW_2\oplus W_1)=2m^2+2m+1,\\
    h^{2,1}(X) & =  \delta(mW_2\tensor ( \Wedge^2 mW_2\oplus m W_2)) +\delta(\Wedge^2 mW_2\oplus m W_2)\\
    & = \delta\left(m W_2 \tensor \left(\tbinom{m+1}{2}W_1\oplus \tbinom{m}{2} W_3\oplus m W_2\right)\right)+ m^2+m \\
    & = m\tbinom{m+1}{2}+2m\tbinom{m}{2}+2m^2+m^2+m\\
    & = (3m+2)\tbinom{m+1}{2}.
\end{align*}
The other Hodge numbers are too cumbersome to compute, so we stop here.
 
\end{exam}
\begin{exam}
 Assume now that the Jordan partition of $B$ is $n=m\cdot 2+1\cdot 1$ (that is, $q_2=m, q_1=1$) and $\epsilon=1$, and choose $j=2$ so that $\gothg$ is indeed $2$-step nilpotent. Therefore $\gotha^*\isom W_1\oplus n W_2$, just as in Example \ref{exam: 1}. Since the Betti numbers depend only on $\gotha^*$, we have that the Betti numbers $b_k(X)$ in this case are given by the same expressions as in Example \ref{exam: 1} (with $n$ odd).

For the computation of the Hodge numbers, we note that Lemma \ref{lem: module structures} gives isomorphisms $\einsnull{\lg^*}\cong (m+1)W_2$ and $\nulleins{\gothb^*}\cong W_1 \oplus mW_2$. It follows from Proposition \ref{prop: betti-hodge} that the first Hodge numbers are
\begin{align*}
    h^{1,0}(X)&=\delta(\einsnull{\lg^*})= m+1,  \\ 
    h^{0,1} (X)& =\delta(\nulleins{\gothb^*})+\delta (\IC)= m+2, \\    
    h^{2,0}(X)&  = \delta(\Wedge^2 \einsnull{\lg^*})=(m+1)^2,  \\
    h^{1,1}(X)&= \delta(\nulleins{\gothb^*}\tensor \einsnull{\lg^*})+\delta(\einsnull{\lg^*})=2(m+1)^2,   \\
    h^{0,2}(X)&= \delta(\Wedge^2\nulleins{\gothb^*})+\delta(\nulleins{\gothb^*})=(m+1)^2.
\end{align*}
\end{exam}

\bibliographystyle{alpha}
\bibliography{references}

\end{document}